\documentclass[a4wide,11pt]{article}
\usepackage[utf8]{inputenc}
\usepackage[pagewise]{lineno}
\usepackage{amsmath,amssymb,amsthm,mathtools}
\usepackage{xcolor}
\usepackage{graphics} 
\usepackage{epstopdf}
\usepackage[colorlinks=true]{hyperref}
\usepackage{tikz}
\usepackage{makeidx}
\usepackage{subfigure}
\usepackage{caption}
\usepackage{setspace}
\usepackage{bbold}
\usepackage{a4wide}
\usepackage{nicefrac}
\usepackage{verbatim}
\usepackage{authblk}
\usepackage{algorithm} 
\usepackage{algpseudocode} 
\theoremstyle{plain}
\newtheorem{theorem}{Theorem}[section]

\newtheorem{lemma}[theorem]{Lemma}
\newtheorem{proposition}[theorem]{Proposition}

\theoremstyle{definition}
\newtheorem{definition}[theorem]{Definition}

\theoremstyle{remark}

\newcommand{\R}{\mathbb{R}}
\newcommand{\N}{\mathbb{N}}
\newcommand{\G}{\mathcal{G}}
\newcommand\norm[1]{\left\lVert#1\right\rVert}
\newcommand{\half}{\frac{1}{2}}

\title{Filtering methods for coupled inverse problems}
\author{Michael Herty\footnote{Institut f\"{u}r Geometrie und Praktische Mathematik,  		RWTH Aachen University,  
		Templergraben~55, 52062~Aachen, Germany, $\{herty,iacomini\}@igpm.rwth-aachen.de$} \mbox{ and } Elisa Iacomini$^*$}               
\date{\today}

\begin{document}

 \maketitle
\begin{abstract}
	We are interested in ensemble methods to solve multi-objective optimization problems. An ensemble Kalman method is proposed to solve a formulation of the nonlinear problem using  a weighted function approach. An analysis of the mean field limit of the ensemble method yields an explicit update formula for the weights. Numerical examples show the improved performance of the proposed method.    
\end{abstract}

\section{Introduction}

In many applications, it is often required to determine the model parameters that approximate observable and noisy data. In this work we are concerned with those inverse problems in a finite dimensional setting, i.e., 
\begin{equation}\label{eq:enkf}
    y=\G(u)+\eta
\end{equation}
where $\G$ is the (possible nonlinear) forward operator between the finite dimensional spaces $X=\R^d$ and $Y=\R^k$ with $d,k\in \N$, $u \in X$ is the unknown parameter, $y\in Y$ is the observation and $\eta \sim \mathcal{N}(0,\Gamma)$ is the observational noise where $\Gamma$ is a known covariance matrix.  Given the noisy measurements, the observation and the mathematical model $\G$, we are interested in finding the corresponding control $u$.  Certainly, those problems have been widely studied and different approaches have been proposed in the literature in order to overcome possible ill--posedness of the problem, see e.g. \cite{engl1996regularization} for a survey.

In this work, we will focus on a particular numerical method for solving \eqref{eq:enkf}, namely the Ensemble Kalman Filter (EnKF). This method was introduced in the last decade \cite{evensen1994sequential}, but has gained  recent attention due to novel developments and insights, see e.g. \cite{herty2019kinetic,herty2020continuous,schillings2017analysis} and references therein. The EnKF aims to solve a least--square formulation of the inverse problem and produces $u^*$ such that 

\begin{equation}\label{eq:phi_sing}
   u^*=argmin_{u \in X} \half \norm{\Gamma^{-\half} (y-\G(u))}^2.
\end{equation}

The EnKF is an iterative filtering method which sequentially updates each member of an ensemble $k=1,\dots,K$ of  elements $u_k$ in the space $X$ by means of the Kalman update formula, using the knowledge of the model $\mathcal{G}$ and  given observational data $y$. The method is gradient free and even for small number of ensembles $K$ satisfactory results have been reported \cite{majda2018performance}. Several contributions have been made regarding the application and analysis of this method, see e.g.  \cite{aanonsen2009ensemble, blmker2018strongly, blomker2019well, chada2020tikhonov, iglesias2015iterative,janjic2014conservation,schillings2017analysis, schwenzer2020identifying} and extensions to the constraint case \cite{chada2019incorporation}.

Here, we are interested in a possible extension of the method towards a multi--objective minimization formulation. Those are also known as coupled inverse problems where for given data, a choice of  parameters for competing models has to be determined. Examples of such problems stem from  applications in geophysics  \cite{kabanikhin2015coupled} to oil and water reservoir problems \cite{sun1990coupled}. We propose a formulation for general multi-objective optimization problems in the forthcoming section using a classical weighted function approach. By extending prior work \cite{herty2019kinetic} we will focus on suitable update strategies for the weights based on a mean field description of the method. Numerical results will be performed to highlight the properties of the proposed method.


\section{On the Ensemble Kalman Filter (EnKF) for Coupled Inverse Problems}\label{sec_1}

We consider  $l$ coupled inverse problems for a set of parameters $u\in X$ and consider the simultaneous minimization 
of $\G_1,\dots,\G_l$ models, given observations $y_1,\dots, y_l \in Y$: 

{\begin{equation}\label{eq:enkfo2}
	\min_{u \in X}  \left( \| \Gamma^{-\half} \left( y_1-\G_1(u) \right) \|, \dots,  \| \Gamma^{-\half} 
	\left( y_l-\G_l(u) \right) \|
	\right).
\end{equation}}
The observational noise on the data $y_i$ is  $\eta_i   \sim \mathcal{N}(0,\Gamma)$ with fixed covariance matrix  $\Gamma$. 
Finding $u$ that simultaneously solves \eqref{eq:enkfo2} is called multi--objective or multi criteria optimization, see e.g.  \cite{ehrgott2005multicriteria, miettinen2012nonlinear,pardalos2017non}.  In the following we use the concept of  Pareto optimality \cite{pardalos2017non} that defines a notion of minimum for the vector--valued optimization problem \eqref{eq:enkfo2}: 

\begin{definition}\label{def1}
	A point $u^*\in \R^d$ is called Pareto optimal if and only if there exists no point $u \in \R^d$ such that $\G_i(u) \le \G_i(u^*)$ for all $i=1,2,\dots,l$ and $\G_j(u)\le \G_j(u^*)$ for at least one $j\in \{1,2,\dots,l\}$.
\end{definition} 
The set $\mathcal{S}_U$ of all  $u^*$ fulfilling Definition \eqref{def1} is called Pareto set, while its representation in the space of objectives  $\mathcal{S}_G := \{ \left(y_i-\G_i(u)\right)_{i=1}^l : u \in \mathcal{S} \}$ is called Pareto front.  An approach based on a weighted function  approach  \cite{miettinen2012nonlinear} is followed to compute $\mathcal{S}_G:$  Given a vector $\mathbf{\mathbb{\lambda}} \in   \Lambda$ where  

\begin{equation}\label{lambda} \Lambda:=\{ \mathbf{\mathbb{\lambda}} \in \R^l_+: \mathbf{\mathbb{\lambda}} \cdot \mathbf{1}=1   \}
\end{equation} and $\mathbf{1}=(1,\dots,1)^T,$
 we define the weighted objective function 
 
\begin{equation}\label{eq:g}
    	\G(u,\mathbf{\mathbb{\lambda}}):= \sum\limits_{i=1}^\ell \mathbf{\mathbb{\lambda}}_{i}  \G_i(u): X\times \Lambda \to Y.
\end{equation}
The convex combination of the observations is given by 
$    y= \sum\limits_{i=1}^\ell \mathbf{\mathbb{\lambda}}_{i}  y_i. $
An approximation to the Pareto front $\mathcal{S}_U$ is then obtained by 
\begin{eqnarray}\label{paretof}
	{P}:=\{ u^*(\mathbb{\lambda}) : \mathbb{\lambda} \in \Lambda  \}, \label{pareto1}
\end{eqnarray}
where 

\begin{equation}\label{opt1}
	u^*(\mathbb{\lambda}) = argmin_{u \in X} \Phi(u,\mathbb{\lambda}), \;  \Phi(u,y,\mathbb{\lambda})=\frac{1}{2} \norm{\Gamma ^{-\half} \sum_{i=1}^l\mathbb{\lambda}_i \left( y_i-\G_i(u) \right)}^2.
\end{equation}


In case of a convex problem, $S_U = P$, {see \cite[Theorem 3.1.4]{miettinen2012nonlinear}.
Note that $\Lambda$ is also called the probability simplex \cite{boyd2004convex}.}


\subsection{EnKF and Mean Field Description of Parameterized Problem \eqref{opt1} }

For the efficient computation of the Pareto front \eqref{pareto1} we propose an ensemble based method following recent work \cite{herty2019kinetic,herty2020continuous,schillings2017analysis}. 
For a fixed value of $\mathbb{\lambda} \in \Lambda$, the EnKF method samples $J>0$ initial values $u^{j,0} \in X$ and iterate according to equation  \eqref{eq:updating} for some $\Delta t>0.$ Under suitable assumptions on $\G$ it has been shown in \cite{schillings2017analysis}, that 

\begin{align}
\lim\limits_{J\to\infty}	\frac1J \sum\limits_{j=1}^J u^{j,n}(\mathbb{\lambda} ) = u^\ast(\mathbb{\lambda} ),
\end{align}
where $u^\ast(\mathbb{\lambda} )$ solves equation \eqref{opt1}, \cite[Theorem 1]{ding2021ensemble}. For further results on stability we refer to \cite{herty2019kinetic,schillings2017analysis}. For $y$ and $\G$ defined by \eqref{eq:g} each member $j$ of the ensemble is propagated according to  

\begin{equation}\label{eq:updating}
    u^{j,n+1}=u^{j,n}+ C(U^n)\left( D(U^n) + \frac{1}{\Delta t} \Gamma^{-1} \right)^{-1}\left[ y -\G(u^{n},\lambda) \right],
\end{equation}
where $C(U^n)$ and $D(U^n)$ are the covariance matrices depending on the  set  of ensembles $U^n(\mathbb{\lambda} )$ at the iteration $n$ and on $\G(U^n)$: 

  \begin{align}\label{eq:c_disc}
    U^n(\mathbb{\lambda} ) &=\{u^{j,n}(\mathbb{\lambda} )\}_{j=1}^J, \\
    \Bar{U}^n&:=\frac{1}{J}\sum_{k=1}^J u^{j,n}(\mathbb{\lambda} ), \; \Bar{\G}:=\frac{1}{J}\sum_{k=1}^J\sum_{i=1}^l \mathbb{\lambda}_i \G_i(u^{j,n}(\mathbb{\lambda}),\mathbb{\lambda} )), \\
    C(U^n(\mathbb{\lambda} ))=& \frac{1}{J}\sum_{k=1}^J(u^{k,n}(\mathbb{\lambda} )-\Bar{U}^n)\otimes(\G(u^{k,n}(\mathbb{\lambda}),\mathbb{\lambda} )-\Bar{\G}), \\
    D(U^n(\mathbb{\lambda} ))=& \frac{1}{J}\sum_{k=1}^J \left[\G(u^n(\mathbb{\lambda}),\mathbb{\lambda} ))- \Bar{\G}\right]\otimes\left[\G(u^n(\mathbb{\lambda}),\mathbb{\lambda} )- \Bar{\G}\right].
\end{align}
Several extensions have been studied and we refer to the references above for more details. Also, the limiting equation for $\Delta t \to 0$ under the scaling $\Gamma^{-1}=\Delta t \Gamma^{-1}$ of the previous dynamics has been studied and analyzed, e.g.  \cite{schillings2017analysis,herty2019kinetic}. In the case $\Delta t \to 0$ and $J\to \infty$ a mean field limit is obtained. Rigorous results can be found e.g. in \cite{herty2019kinetic} and in \cite{pareschi2013interacting,carrillo2010particle} for general mean field results on interacting particle systems. Since there is no dynamics in $\lambda$ the following result is a simple consequence of the existing results for the convergence for $J\to \infty$ given e.g. in  \cite{carrillo2021wasserstein,ding2021ensemble,garbuno2020interacting,herty2019kinetic,herty2020continuous,schillings2017analysis}, in particular \cite[Theorem 3]{schillings2017analysis}.

\begin{proposition} \label{prop}
Assume $\G_i(u)=G_i u$ for $i=1,\dots, l$  and let $\Phi$ be given by equation \eqref{opt1}.   
Let $\mathcal{P}(X)$ be the space of probability measures on $X$ equipped with the 1-Wasserstein distance. 

Let  $J>0$ and assume  $u^{j,0} \in X$ for $j=1,\dots, J$ given and denote by $f^U_0(v)=\frac1J \sum_{j=1}^J  \delta \left(u^{j,0}-v \right)$ the empirical measure associated to the initial data.  
 The empirical 
measure   

\begin{align}
	f^U(v,\mathbb{\lambda},t) =\frac1J \sum_{j=1}^J  \delta \left(u^j(t,\mathbb{\lambda})-v \right) \in \mathcal{P}(X \times \Lambda \times \R^+) 
\end{align}
where  $u^j(t,\mathbb{\lambda})$ fulfills  for all $j=1,\dots,J$ 

\begin{align}
	\frac{d}{dt} u^j(t,\mathbb{\lambda}) &=-  C(U(t,\mathbb{\lambda})) \nabla \Phi(u^j(t,\mathbb{\lambda}),y,\mathbb{\lambda}) , \;  	 u^j(t,\mathbb{\lambda}) =u^{j,0},  \\
	\mathcal{C}(U) &=\frac{1}{J}\sum_{j=1}^J (u^j-\Bar{U})\otimes(u^j-\Bar{U}), \; \bar{U}=\frac1J\sum\limits_{j=1}^J  u^j,  
\end{align}
	is a solution   in the distributional sense to the mean field equation 
	
\begin{align}\label{meanfield}
\partial_t f(v,\mathbb{\lambda},t)-\nabla_v \cdot \left( \mathcal{C}(t,\mathbb{\lambda})\nabla_v \Phi(u,y,\mathbb{\lambda})f(v,\mathbb{\lambda},t) \right)=0, \; f(v,\lambda,0)=f_0(v), 
\end{align}
subject to  initial data $ f_0(v,\mathbb{\lambda}) \in \mathcal{P}(X,\Lambda)$ and where the nonlocal operator $C(t,\mathbb{\lambda})=C[f](t,\mathbb{\lambda})$ is given by

\begin{align}\label{C}
	(\mathcal{C}[f](\mathbb{\lambda},t))_{k,i} =\int_{X} v_k v_i f(v,\mathbb{\lambda},t)\, dv - \int_{X} v_k f(v,\mathbb{\lambda},t)\, dv  \int_{X} v_i f(v,\mathbb{\lambda},t)\, dv , \; (k,i)=1,\dots, d.
\end{align}     
Furthermore, if for $J\to \infty$ we have $W_1(f^U_0,f_0) \to 0 $ for some $f_0 \in \mathcal{P}(X),$ then for any $t\geq 0$ 
we have $W_1(f^U(\cdot,t), f(\cdot,t)) \to 0,$ where $f$ is a solution in the distributional sense to \eqref{meanfield}. 
\end{proposition}

Denote by $P(t)$ an approximation to $P$ expressed in terms of the probability density $f(\cdot,t)$ by

\begin{equation}
	\label{P2} 
		{P}(t)=\left\{ \int_X u \; d f(u,\mathbb{\lambda},t): \;   \mathbb{\lambda} \in \Lambda  \right\}. 
\end{equation}
Due to the convergence of the particles to $u^\ast(\lambda)$, we expect that for $t\to \infty,$ the set $P(t)$ approaches the 
set $P$ given by  \eqref{pareto1}, see \cite{herty2020continuous} for the corresponding result in the case independent of $\lambda.$ 
\par 
The mean field equation is independent of the ensemble size $J$ and therefore possibly attractive for numerical methods. For an efficient computation of $P(t)$ the solution to equation \eqref{meanfield} for any $\lambda \in \Lambda$ is required. In numerical discretization of equation \eqref{meanfield} a suitable grid in $\lambda$ is hence necessary. In the following aim to provide a method to develop a strategy for choosing those quadrature points in $\Lambda.$ This is obtained by considering the sensitivity of $f$ with respect to $\lambda.$ 

\subsection{Sensitivity of Mean Field and Moment Equations }\label{sec_3}

The sensitivity of $f$ with respect to $\lambda$ can be studied e.g. by  formally differentiating the meanfield equation \eqref{meanfield} 
leading to the set of equations   $i=1,\dots,l$ 

\begin{align}\label{eq:der}
    0 &=   \partial_t \partial_{\lambda_i} f(v,\lambda,t) - \nabla_v \partial_{\lambda_i} \Big( C\ \nabla_v \Phi(v,y,\lambda) f(v,\lambda, t) \Big) \\
      & =  \partial_t \partial_{\lambda_i} f(v,\lambda,t) - \nabla_v \Big( \partial_{\lambda_i} C \nabla_v \Phi(v,y,\lambda) f(v,\lambda, t) + \\ & \qquad C\ \partial_{\lambda_i} (\nabla_v  \Phi(v,y,\lambda)) f(v,\lambda, t) + C\ \nabla_v  \Phi(v,y,\lambda) \partial_{\lambda_i} f(v,\lambda, t) \Big) \\
      & =:  \partial_t \partial_{\lambda_i} f(v,\lambda,t) - \nabla_v \Big(  T_1 + C T_2 f(v,\lambda,t) + C\ \nabla_v  \Phi(v,y,\lambda) \partial_{\lambda_i} f(v,\lambda, t) \Big).
\end{align}
where $C=C[f](\lambda,t)$ is given by equation \eqref{C}. Since the initial data in \eqref{meanfield} is assumed to be independent of $\mathbb{\lambda}$ we obtain for $i=1,\dots, l$

\begin{align}\label{der-in}
	\partial_{\lambda_i} f(v,\lambda,0)=0.
	\end{align}
 Under the assumption of Proposition \ref{prop}, namely, 
 
 \begin{align}\label{ass1}
 	\G_i(u) = G_i u,  \; i =1,\dots, l, 
 \end{align}
 some terms of equation \eqref{eq:der} can be further simplified to 
 
\begin{align}
    T_1:=&\partial_{\lambda_i} C\ \nabla_v \Phi(v,y,\lambda) f = (E_{\lambda_i}-2 m_{\lambda_i} \otimes m )\ \nabla_v \Phi(v,y,\lambda) f, \\
	T_2:=& \frac{\partial (\nabla_v \Phi)}{\partial \lambda_i} =\partial_{\lambda_i} \Big( (\G)^T \Gamma^{-1} (y-\G v) \Big) \\
	&= (\sum_{i=1}^l \G_i)^T  \Gamma^{-1}\Big(y-\G v\Big)+
	\left( \G \right)^T \Gamma^{-1} \Big( \sum_{i=1}^l y_i - \G_i v  \Big).
\end{align}


Computationally solving  system \eqref{def1} to obtain sensitivity information on $\lambda$ is prohibitive. However, $P(t)$ given by equation \eqref{P2} only depends on the first moment of $f$ and not on the full solution. Hence, we consider only sensitivity of the moments  of the solution, i.e., define the first and second moments of $f$ as

\begin{equation}\label{eq:moments}
	m(\mathbb{\lambda},t)  =\int_{X} v  d f(v,\mathbb{\lambda},t) \in \R^d, \; 
	E(\mathbb{\lambda},t) =\int_{X} v \otimes v d f(v,\mathbb{\lambda},t) \in \R^{d\times d}.
\end{equation}
Then, for $i=1,\dots, l$, the sensitivity of $(m,E)$ is given by 

\begin{align}
	 \frac{\partial m}{\partial \lambda_i} =\int v d \frac{\partial f}{\partial \lambda_i}(v,\lambda,t), \; 
	\frac{\partial E}{\partial \lambda_i} =\int v\otimes v d \frac{\partial f}{\partial \lambda_i}(v,\lambda,t) 
\end{align}
and they fulfill a closed coupled system of equations of ordinary differential equations obtained by integration of equation \eqref{eq:der} given by the system \eqref{eq:mom_eq}.     

\begin{lemma}
	Assume \eqref{ass1} and let $\Phi$ be given by equation \eqref{opt1}. 	If  $f=f(v,\mathbb{\lambda},t) \in \mathcal{P}(X \times \Lambda \times \R^+)$ with finite second moment and $\nabla_\lambda f$ be  differentiable solution  to \eqref{meanfield} with initial data $f(v,\lambda,0) = f_0(v) \in \mathcal{P}(X)$ and to equation \eqref{eq:der} with initial data \eqref{der-in}.
	\par 
	Then, the moments $(m,E)$ and their derivatives fulfill the following system of ordinary differential equations for $i=1,\dots, l$ 
	
\begin{align}\label{eq:mom_eq}
\begin{cases}
  \frac{d}{dt} m &= - CG^T \Gamma^{-1}\left( y-G\ m \right) \\
\frac{d}{dt} m_{\lambda_i} &= - (\partial_{\lambda_i} C)
    G^T \Gamma^{-1} (y-G\ m)
-C  (\partial_{\lambda_i} G )^T \Gamma^{-1} ( y- G m) \\
&- C G^T \Gamma^{-1}\left( \partial_{\lambda_i} y - (\partial_{\lambda_i} G)m \right)
+ C G^T \Gamma^{-1}  G  m_{\lambda_i} \\
\frac{d}{dt}E &= - C G^T \Gamma^{-1} (y \otimes m - G E) - [ C G^T \Gamma^{-1} (y \otimes m - G E)]^T\\
 \frac{d}{dt} E_{\lambda_i} &=-CG^T\Gamma^{-1}\left(y\otimes m_{\lambda_i}-GE_{\lambda_i} \right) - \bar{C}_i G^T\Gamma^{-1}\left( y \otimes m-GE \right)\\
    &-C\partial_{\lambda_i} G^T \Gamma^{-1} (y\otimes m- GE)-CG^T\Gamma^{-1}(\partial_{\lambda_i} y m-\partial_{\lambda_i} G E)\\
    & -[CG^T\Gamma^{-1}\left(y\otimes m_{\lambda_i}-GE_{\lambda_i} \right)]^T - [\bar{C}_i G^T\Gamma^{-1}\left( y \otimes m-GE \right)]^T\\
    &-[C\partial_{\lambda_i} G^T \Gamma^{-1} (y\otimes m- GE)]^T-[CG^T\Gamma^{-1}(\partial_{\lambda_i} y m-\partial_{\lambda_i} G E)]^T\\
   C &= E - m \otimes m, \\
   {\bar C}_i &= E_{\lambda_i} - 2 m_{\lambda_i} \otimes m.
\end{cases}
\end{align}

and initial data independent of $\lambda$

\begin{align}\label{initial-mh}
	m(0)=\int_X v d f_0(v) , \; E(0)=\int_X  v\otimes v  d f_0(v) , \;  m_{\lambda_{i}}(0)=0, \;  E_{\lambda_{i}}(0)=0.
\end{align}
For any time $T>0,$ there exists a unique solution $(m,E,\partial_{\lambda_1} m,  \partial_{\lambda_1} E, \dots, 
,\partial_{\lambda_l} m,  \partial_{\lambda_l} E, ) \in C^1(0,T; \R^{(l+1) \cdot ( d+d\times d)})$ 
to the system \eqref{eq:mom_eq}. 
\end{lemma}
\begin{proof}
The right hand side of \eqref{eq:mom_eq} is Lipschitz with respect to $(m,m_{\lambda_i},E,E_{\lambda_i})$ which yields the existence and uniqueness of the moments. The derivation of the moment system is given by integration based on the formal  equation  \eqref{eq:der}. For simplicity, we assume in the following proof that 
$f$ is absolutely continuous with respect to the Lebesgue measure. We denote the induced density also by $f.$ 
\par 
First, note that since $\partial_{\lambda_i} f$ is a conservative equation with initial data \eqref{der-in} and therefore 

\begin{align}
	\int_X \partial_{\lambda_i} f(v,\lambda,t) d v = 0.
\end{align}
Second, the first and the third equation of  system \eqref{eq:mom_eq} follow immediately by integration of the mean field equation \eqref{meanfield}. Indeed for the third equation we obtain

\begin{equation}
	\partial_t \int_X v_i v_j f\ dv-  \sum_{k=1}^d \int v_i v_j \partial_k (C \nabla_v \Phi(v,y,\lambda) f)_k \ dv =0,  \quad i,j=1,\dots,d,
\end{equation} 
and, integrating by parts

\begin{equation}
\partial_t E_{i,j} + \sum_{k=1}^d \int \partial_k(v_i v_j) (C \nabla_v \Phi(v,y,\lambda) f)_k \ dv =0.
\end{equation} 
Thus

\begin{align}
\partial_t E_{i,j} + \int  \left[  (C \nabla_v \Phi(v,y,\lambda))_i v_j f + (C \nabla_v \Phi(v,y,\lambda))_j v_i f \right] dv =0,\\
\partial_t E_{i,j} +\sum_{l=1}^d (C G^T \Gamma^{-1})_{i,l}y_l m_j - \sum_{l=1}^d (C G^T \Gamma^{-1}G)_{i,l}E_{l,j} \\+ \sum_{l=1}^d (C G^T \Gamma^{-1})_{j,l}y_l m_i - \sum_{l=1}^d (C G^T \Gamma^{-1}G)_{j,l}E_{l,i}=0.
\end{align}
Hence,  we obtain  equation \eqref{eq:mom_eq}

\begin{equation}
\partial_t E + C G^T \Gamma^{-1} (y\otimes m - GE) + \left[C G^T \Gamma^{-1}(y\otimes m - GE )\right]^T=0.
\end{equation}
Since the  operator $C$ is linear in  $f$ we obtain  

\begin{equation}
	\frac{\partial C}{\partial \mathbb{\lambda}_i} = E_{\lambda_i}-2 m_{\lambda_i} \otimes m, 
\end{equation}
and similar to term $T_2:$

\begin{align}\label{eq:grad_phi_lambda}
    \frac{\partial (\nabla_v \Phi)}{\partial \mathbb{\lambda}_i}&=
     (\sum_{i=1}^l G_i)^T  \Gamma^{-1}\Big(y-Gv\Big)+
   \left( G \right)^T \Gamma^{-1} \Big( \sum_{i=1}^l y_i - G_i v  \Big). 
\end{align}
Hence, integration of  \eqref{eq:der} yields 

\begin{align}
    \partial_t m_{\lambda_i} + \Big(E_{\lambda_i}-2 m_{\lambda_i} \otimes m\Big)
    \int_{X}\nabla_v  \Phi(v,y,\mathbb{\lambda}) f\ dv +\\
C \int_{X} \Big( (\sum_{i=1}^l G_i)^T  \Gamma^{-1}\Big(y-Gv\Big)+
   \left( G \right)^T \Gamma^{-1} \Big( \sum_{i=1}^l (y_i - G_i v)  \Big) f \, dv \label{eq:sec}
\\
+ C \int_X \nabla_v \Phi(v,y,\mathbb{\lambda}) \partial_{\mathbb{\lambda}_i} f  \ dv=0.  
\end{align}
Due to assumption \eqref{ass1} the integrals involving $\nabla_v \Phi$ are computed explicitly

\begin{align}
  \int_{X} \nabla_v \Phi(v,y,\mathbb{\lambda}) f  \ dv &= \int_{\R^d} G^T \Gamma^{-1} (y-Gv) f \ dv  = G^T \Gamma^{-1}\left( y-G\ m \right), \\
  \int_{X} \nabla_v \Phi(v,y,\mathbb{\lambda})\partial_{\mathbb{\lambda}_i} f  \ dv & = G^T \Gamma^{-1}\left( y\int_{\R^d} \partial_{\mathbb{\lambda}_i} f  dv - G\int_{\R^d} v\partial_{\mathbb{\lambda}_i} f dv \right) =  + G^T \Gamma^{-1} G\ m_{\lambda_i}. 
\end{align}
Furthermore, we obtain 

\begin{align}
C& \int_{X} \Big( (\sum_{i=1}^l G_i)^T  \Gamma^{-1}\Big(y-Gv\Big)+
   \left( G \right)^T \Gamma^{-1} \Big( \sum_{i=1}^l (y_i - G_i v)  \Big) f \, dv\\
  =&C  (\sum_{i=1}^l G_i)^T \Gamma^{-1} ( y- G m) + C \left( G\right)^T \Gamma^{-1}\left( \sum_{i=1}^l (y_i - \G_i m ) \right),
\end{align}
leading to the  equation for $m_{\lambda_{i}}.$ The equations for  $\frac{d}{dt}E_{\lambda_i}$ are obtained using a similar computation. 
\end{proof}

Some remarks are in order. 
\begin{itemize}
	\item Note that the approximation to the Pareto front $P(t)$ on the mean field level 
is given by
 
\begin{align}
	P(t)=\{ m(\lambda,t): \lambda \in \Lambda \}. 
\end{align}
Hence, solving a coupled system of ordinary differential equations of dimension $(l+1)\times (d+d^2)$  leads to information on $\nabla_\lambda m(t).$  This allows to obtain information for  an adaptive strategy for the choice of $\lambda$ as follows: 
Assume for a fixed $\overline{\lambda},$ the optimal state is given by $m(\overline{\lambda},T)$ for some $T$ fixed and sufficiently large. Then, we may use a Taylor expansion of $m$ to obtain 

\begin{align}\label{formula}
	m(\overline{\lambda}+\Delta \lambda,T) = m(\overline{\lambda},T) + \Delta \lambda \cdot \nabla m(\overline{\lambda},T) + h.o.t.,
\end{align}
where $\nabla m(\lambda,t)=\left( m_{\lambda_i} \right)_{i=1}^l$.  The previous expansion can be used in two ways:  For a given update $\Delta \lambda \in \R^l$ such that $\overline{\lambda} + \Delta \lambda \in \Lambda$, equation \eqref{formula} yields  an approximation on the new optimal value of the Pareto front $P(t)$.  Second, we observe that the system \eqref{eq:mom_eq} can be solved independently of the dynamics of $f=f(v,\lambda,t)$ leading to a family of solutions for $\lambda \in \Lambda$ and $t \geq 0$ 

\begin{align}
	\left( m(\lambda,t), \; \nabla m(\lambda,t) \right),
\end{align}
that can be computed a priori.  We are interested in obtaining a discrete choice of $\lambda^k \in \Lambda$ for $k=1,\dots,K$ such that the Pareto set $S_U= \{ u^*(\lambda): \lambda \in \Lambda \}$ is  approximated without clustering. Since for $T$ large we have 
$m(\lambda,T) \approx u^*(\lambda)$ we may utilize equation \eqref{formula} to determine at least the norm of the update $\Delta \lambda = \lambda^{k+1}-\lambda^k$ such that the distance on $S_U$ is bounded by a given tolerance $\delta>0$ by requiring

\begin{align} \label{update}
	\| \Delta \lambda \| \| \nabla m(\lambda^k ) \| \leq \delta.
\end{align}
 This choice  leads to  numerical results shown later that also approximates the Pareto front $S_G$ very well with only a few discretization points $k=1,\dots,K$ in $\Lambda.$  

\item The convergence results on the EnKF require usually $n$ or $t$, respectively to tend to infinity. In the particular situation where the system of ordinary differential equations allows for  steady--state solutions $(m,E, \nabla m, \nabla E)$, this value is therefore expected to be also a solution to the Pareto problem. The following equations characterize the steady--state solutions to \eqref{eq:mom_eq} for $i=1,\dots,l$  only in the case {\bf d=1}: 

\begin{align}
	&G m={y}, \qquad &m_{\mathbb{\lambda}_i}=c_{1,i}, \qquad &G^2 E={y^2}, \qquad &E_{\mathbb{\lambda}_i} =c_{2,i}\label{eq:null_1}, \\
	&m=c_1, \qquad &m_{\mathbb{\lambda}_i} =c_{1,i}, \qquad &E=m^2, &E_{\lambda_i}=2m \ m_{\mathbb{\lambda}_i}. \label{eq:null_2} 
\end{align}
where, $c_{1,i}$ and $c_{2,i}$ are arbitrary constants. Note that if $(m,E)$ is a set of moments of an underlying distribution function $f_\infty(u,\lambda),$ then by definition we obtain that $E\geq m^2$ imposing restrictions on the set of admissible constants $c_{k,i}$ for $k=1,2$ and $i=1,\dots,l,$. 

\item The case ${\bf d=1}$  also allows for an explicit computations of the Pareto front are possible, provided that the operator 
$\G= \sum\limits_{i=1}^l \lambda_i G_i u: X \to Y$ is invertible. 
In this case, the true solution is given by 

\begin{align}
	u^*(\mathbb{\lambda})=\G^{-1}y
\end{align}
and on the mean field level, we expect $f(v,\lambda,t)=\delta(v-u^*(\mathbb{\lambda}))$ to be the stationary solution. In fact, the following computation verifies 
that $f$ is a stationary state of the moment system \eqref{eq:mom_eq}. Note that this particular probability measure $f$ defines a distribution on the set of functions   $\psi \in C^\infty_0(X)$ by 

\begin{align}
f[\psi]:=	\int_X \psi(u) d f(u,\lambda,t)  = \psi(u^*(\mathbb{\lambda})). 
\end{align}
Hence, for $\psi(u)=u$ we obtain $G m(\lambda,t) = y$ and for $\psi(u)=u^2$ we have $G^2 E(\lambda,t)=y^2.$ Assuming that $\lambda \to u^*(\lambda)$ is differentiable with respect to $\lambda,$ the weak derivative is 

\begin{align}
f_{\mathbb{\lambda}_i} [\psi]  = \psi'(u^*(\mathbb{\lambda})) u^*_{\lambda_i}(\mathbb{\lambda}), 
\end{align}
and hence, 
$m_{\lambda_i} = u^*_{\lambda_i}(\mathbb{\lambda})$ and $E_{\lambda_i} = 2 m m_{\lambda_i}^2$.  Since $\G u^*(\lambda) =y$ we obtain that $\partial_{\lambda_i} \G \; u^*(\lambda)  =  -  \G u_{\lambda_i}^*(\lambda) + \partial_{\lambda_i} y$ leading to the equality $\partial_{\lambda_i} \G \; m = - \G  m_{\lambda_i} + \partial_{\lambda_i} y$. Hence, it is a steady state of equation \eqref{eq:mom_eq}. 
\end{itemize}

\section{Computational Results}\label{sec_5}
For a numerical solution to the approximation of the Pareto front $S_U$ and $S_G,$ respectively, we compare two strategies. In the  direct approach we sample on an equidistant grid on $\Lambda$ the values of $\lambda^k$. In  an adaptive strategy the solution to the mean field moment system \eqref{eq:mom_eq} is utilized. Without loss of generality, in the numerical tests we assume $l=2$, so that $\G(u)=\mathbb{\lambda} \G_1(u) + (1-\mathbb{\lambda})\G_2(u)$ and such that 
$\Lambda$ is parameterized by a single parameter $\lambda \in [0,1].$  Moreover, we set $y=0$, $\eta=0$, $\Gamma=\mathbb{1}$ and $T_{fin}=10$ for all computations. To solve \eqref{eq:mom_eq}, we use a Matlab function \textsc{ode45}  and  initial data  recovered from the ensemble particles, i.e.,  $m_0=\frac{1}{J}\sum_{j=1}^J u_j$, $E_0=\frac{1}{J}\sum_{j=1}^J u_j^2$, $m_{\lambda_i, 0}=0$, $E_{\lambda_i,0}=0$.

\par 

Even so the theory is presented in the linear case only, we present numerical results on nonlinear objective functions $\G_i$ in the numerical tests. Note that the existing literature on convergence and stability of the EnKF do not cover the nonlinear case, even in the case of only finitely many particles. 
Numerically, we propose two possible strategies to adapt method to the nonlinear case. In the first case and if the derivative of $\G$ is computable, we may linearize \eqref{opt1} up to the first order: 

 \begin{align} 
	\norm{y-\G(u)} \approx \norm{y-\G(u_0)+\G'(u_0)u_0-\G'(u_0)u}=\norm{\tilde{y}-\G'(u_0)u}. 
\end{align}
Replacing the nonlinear objective by its linearized version allows to apply the aforementioned results. However, an advantage of the EnKF is that it also applies to functions where no derivative information is available. Therefore, we secondly, consider $\G(m)$ instead of $G \; m$ in system \eqref{eq:mom_eq}. This simple heuristic modification is not justified by a moment analysis, since, in fact, the moment system in the nonlinear case is not closed.

\subsection{Direct approach}

Starting from an initial ensemble $u^0_j$ for $j=1,\dots,J$ and a set of fixed  vectors $\mathbb{\lambda}_k \in \Lambda$ for 
$k=1,\dots,N_{\lambda}$,  the  particles are updated following \eqref{eq:updating}.  As in \cite{deb2019generating} we chose the vectors  to be equispaced. The algorithm is described in detail in Figure \ref{alg_direct}.

 \begin{algorithm}
	\caption{Direct approach} \label{alg_direct}
	\begin{algorithmic}[1]
	\State Given $J$ samples $u_j^{0}$, with $j=1,\dots,J$ and a vector $\mathbb{\lambda}_{0_i}$, $i=0,\dots, l$
	\State Set $n=0$, $t^0=0$ and final time $T_{fin}$ sufficiently large
		\For {$k=1,2,\ldots,N_\mathbb{\lambda}$}
			\State Solve the EnKF procedure:
			$\G= \mathbb{\lambda}_{k_i} \cdot \G_i$,	$y=\mathbb{\lambda}_{k_i} \cdot y_i $
			\While{$t^n \le T_{fin}$} 
			\begin{align*}	
u_j^{n+1}&=u_j^{n}+ C(\boldsymbol{u}^n)\left( D(\boldsymbol{u}^n) + \frac{1}{\Delta t} \Gamma^{-1} \right)^{-1}\left[ y_j -\mathcal{G}({u}_j^{n}) \right]\\
C(\boldsymbol{u}^n)&= \frac{1}{J}\sum_{j=1}^J(u_j^{n}-\overline{\boldsymbol{u}}^n)\otimes(\mathcal{G}(u_j^{n})-\overline{\mathcal{G}}) \tag{EnKF}\\
D(\boldsymbol{u}^n)&= \frac{1}{J}\sum_{j=1}^J \left[\mathcal{G}(u_j^{n})- \overline{\mathcal{G}}\right]\otimes\left[\mathcal{G}(u_j^{n})- \overline{\mathcal{G}}\right]
\end{align*}
\EndWhile
			\State The mean $ \frac1J\sum\limits_{j=1}^J u^{T_{fin}}_j$ is an approximation to $u^*(\lambda)$ 
		\EndFor

	\end{algorithmic} 
\end{algorithm}

\subsection{Adaptive strategy}

In the adaptive strategy  the vector $\mathbb{\lambda}_k$ is obtained iteratively for $k=1,\dots$ according to equation \eqref{update}. Intuitively, the equation yields a denser set of vectors $\lambda_k$ where the slope of the Pareto set $S_U$ measured through $\| \nabla m(\lambda_k) \|$ is large.   In order to state the  update formula an ordering on $\Lambda$ is introduced as  lexicographic order on the set $\Lambda$. The adaptive strategy using the update  given by equation \eqref{update} is given below in Figure \ref{alg_adapt}. 

\begin{algorithm}
	\caption{Adaptive approach}  \label{alg_adapt}
	\begin{algorithmic}[1]
	\State Given $J$ samples $u^{j,0}$, with $j=1,\dots,J$ and the update constant $\delta>0$
	\State set $n=0$, $t^0=0$, the final time $T_{fin}$ and $\mathbb{\lambda}_1=0$
		\While{$\mathbb{\lambda}_k<1$}
			\State $\Bar{u}_k \leftarrow$ solving the EnKF procedure (as in Step 5 of the Direct Approach)
			\State $[m,m_\mathbb{\lambda},E,E_\mathbb{\lambda}]\leftarrow$solve the ODE system \eqref{eq:mom_eq}  
			with initial conditions \eqref{initial-mh}
			\State $\Bar{u}^0_{k+1} \leftarrow$ {sampling from a Gaussian prob. distr. with mean $m$ and variance $E$}
			\State { $\mathbb{\lambda}_{k+1}\leftarrow \mathbb{\lambda}_k + \frac{\delta}{\norm{m_\mathbb{\lambda}}}e_{lo}$, $e_{lo}$ the direction defined by the lexicographic order}
			\EndWhile
	\end{algorithmic} 
\end{algorithm}

\subsection{Test 1: Convex Example}
As numerical test we consider the minimization of two convex functions $\G_1,\G_2$:

\begin{equation}
    \G_1=\left(u-\frac{1}{2}\right)^2 \qquad \G_2=\left(u+\frac{1}{2}\right)^2.
\end{equation}
The initial ensemble is chosen using the uniform distribution $U_0\sim \mathcal{U}(-1,1)$ and we use $J=20$. A comparison of the direct and the adaptive algorithm with $\delta=10^{-3}$ and  $N_\mathbb{\lambda}=25$ is presented. The approximation of the Pareto front $S_G$ is shown in Fig.\ \ref{fig:test1}. 

\begin{figure}[h!]
     \hspace{-1.9cm}
    \includegraphics[scale=0.45]{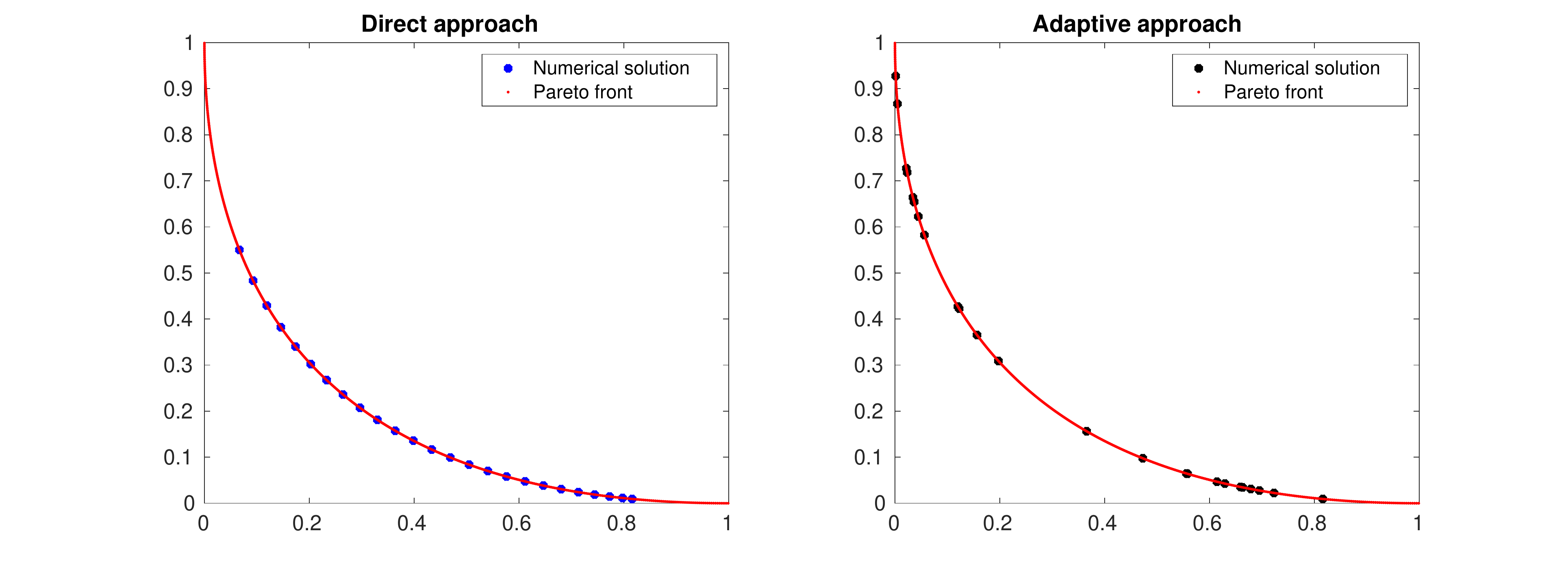}
    \caption{Test 1. Numerical approximation of the Pareto front, with the direct approach (left) and the adaptive approach (right). The red line is the analytical Pareto front. The dots indicate the mean of the ensemble at final time.}
    \label{fig:test1}
\end{figure}
Different behavior of the two procedures is observed, where  the solution obtained by the adaptive approach covers a larger percentage of the Pareto front. Moreover, in Fig.\ \ref{fig:test1_lambda}, we show the distribution of $\mathbb{\lambda}$ in the interval $[0,1]$ for the direct approach (left) and the adaptive one (right). This reflects the fact, that in the adaptive approach a varying grid on $\Lambda$ is obtained according to the update formula \eqref{update}. This simulation validates the intuitive interpretation of this equation.  

\begin{figure}[h!]
   \hspace{-1cm}
    \includegraphics[width=0.55\linewidth]{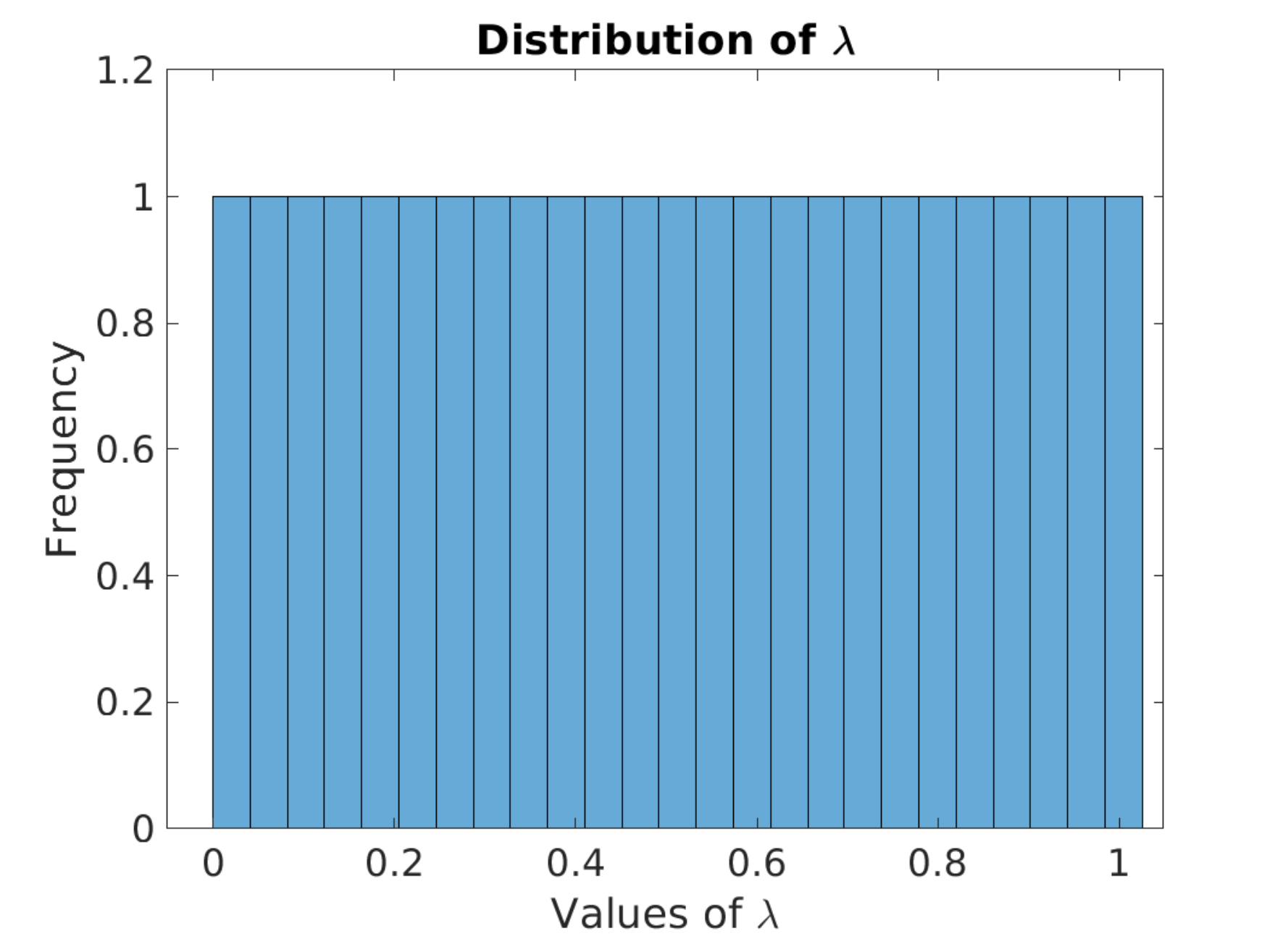}
    \includegraphics[width=0.54\linewidth]{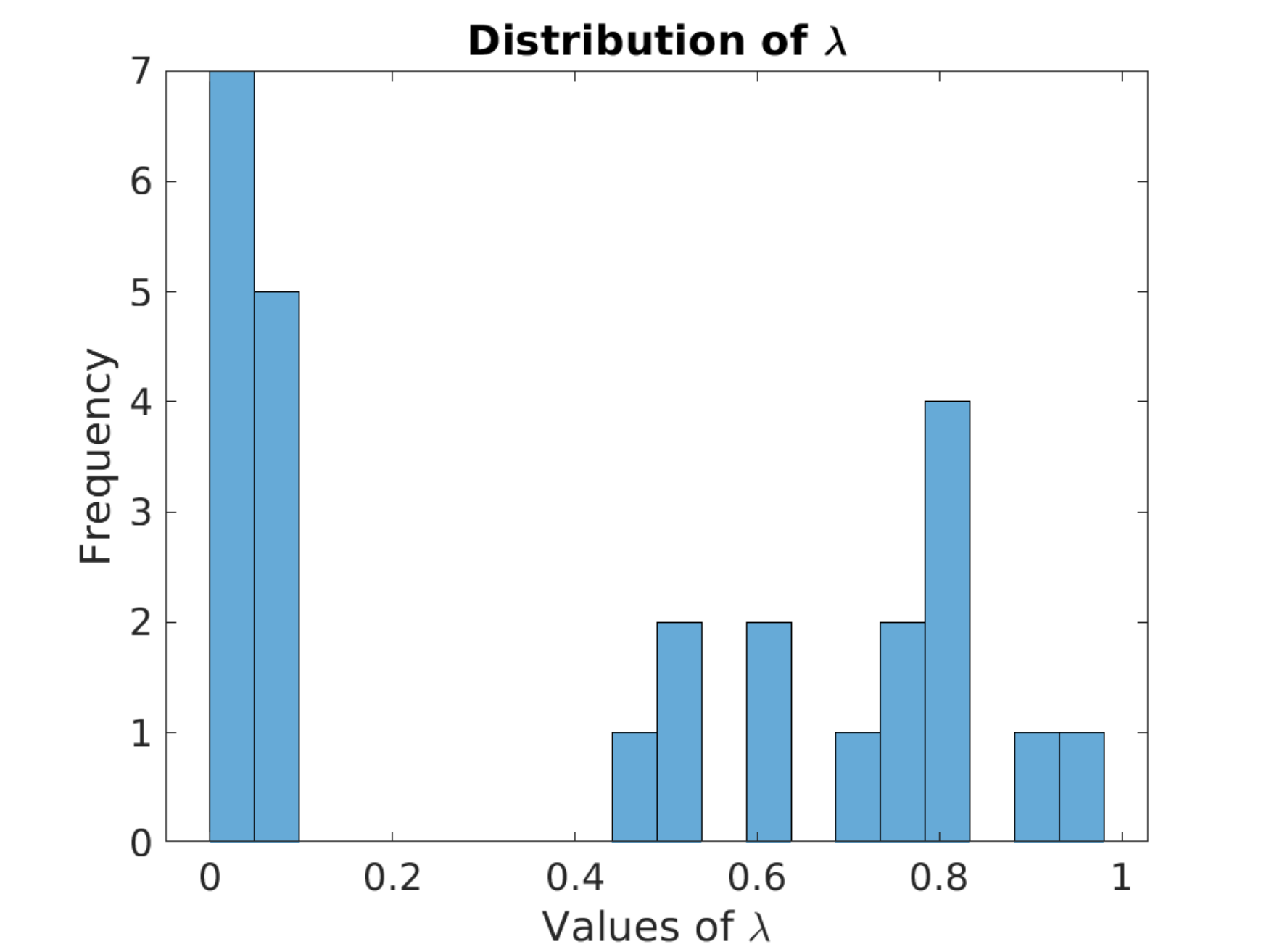}
    \caption{Test 1. Distribution of the sampled values $\mathbb{\lambda}$. }
    \label{fig:test1_lambda}
\end{figure}
A similar behavior is obtained when the discretization in $\lambda$ is refined. In Fig. \ref{fig:test1_2}  the updating constant $\delta$ is $\delta=10^{-4}$  and   $N_\mathbb{\lambda}=54$. 

\begin{figure}[h!]
 \hspace{-2.2cm}
 \includegraphics[scale=0.45]{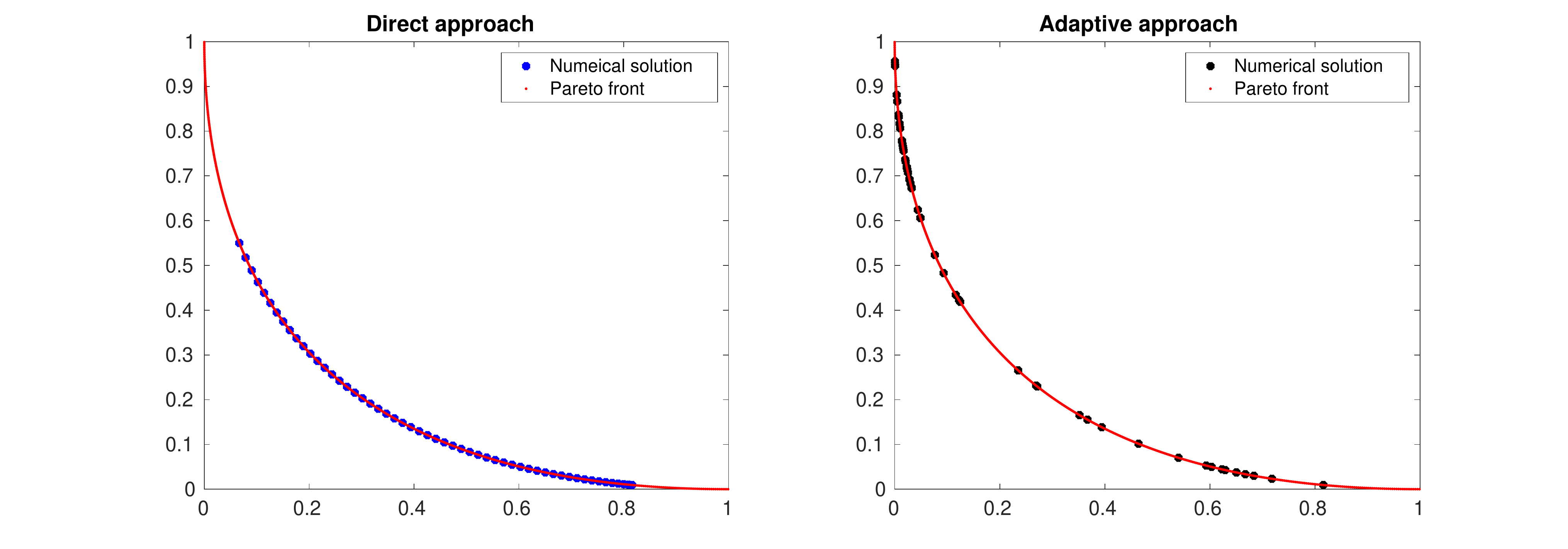}
    \caption{Test 1. Numerical approximation of the Pareto front, with the direct approach (left) and the adaptive approach (right). The red line is the exact Pareto front.} 
    \label{fig:test1_2}
\end{figure}
Focusing on the direct approach, Fig.\ \ref{fig:test1_2} (left), we notice that even for a larger number of values $N_\lambda$    the whole Pareto front is not covered. 


This graphical interpretation is also  compared quantitatively. Given a parametrization of the Pareto front and an equispaced grid, we consider the sum of the minimal distance $d$ between each point of the grid, $x_i$ for $i=1,\dots,N_g$ with $N_g>0$, and the mean of the ensembles at terminal time  for different values of  $N_\mathbb{\lambda}$

\begin{equation}\label{eq:distance}
    \frac{\sum_{i=1}^{N_g} \min d(x_i,u^*(\mathbb{\lambda}))}{N_g}
\end{equation}
The measure \eqref{eq:distance} is similar to the notion of performance metric (IGD) described in \cite{zhang2008multiobjective}. The comparison shows the improved performance of the adaptive approach for increasing number $N_\lambda$ as expected, see Fig.\ \ref{fig:test1_distance}.

\begin{figure}
    \centering
    \includegraphics[width=0.5\linewidth]{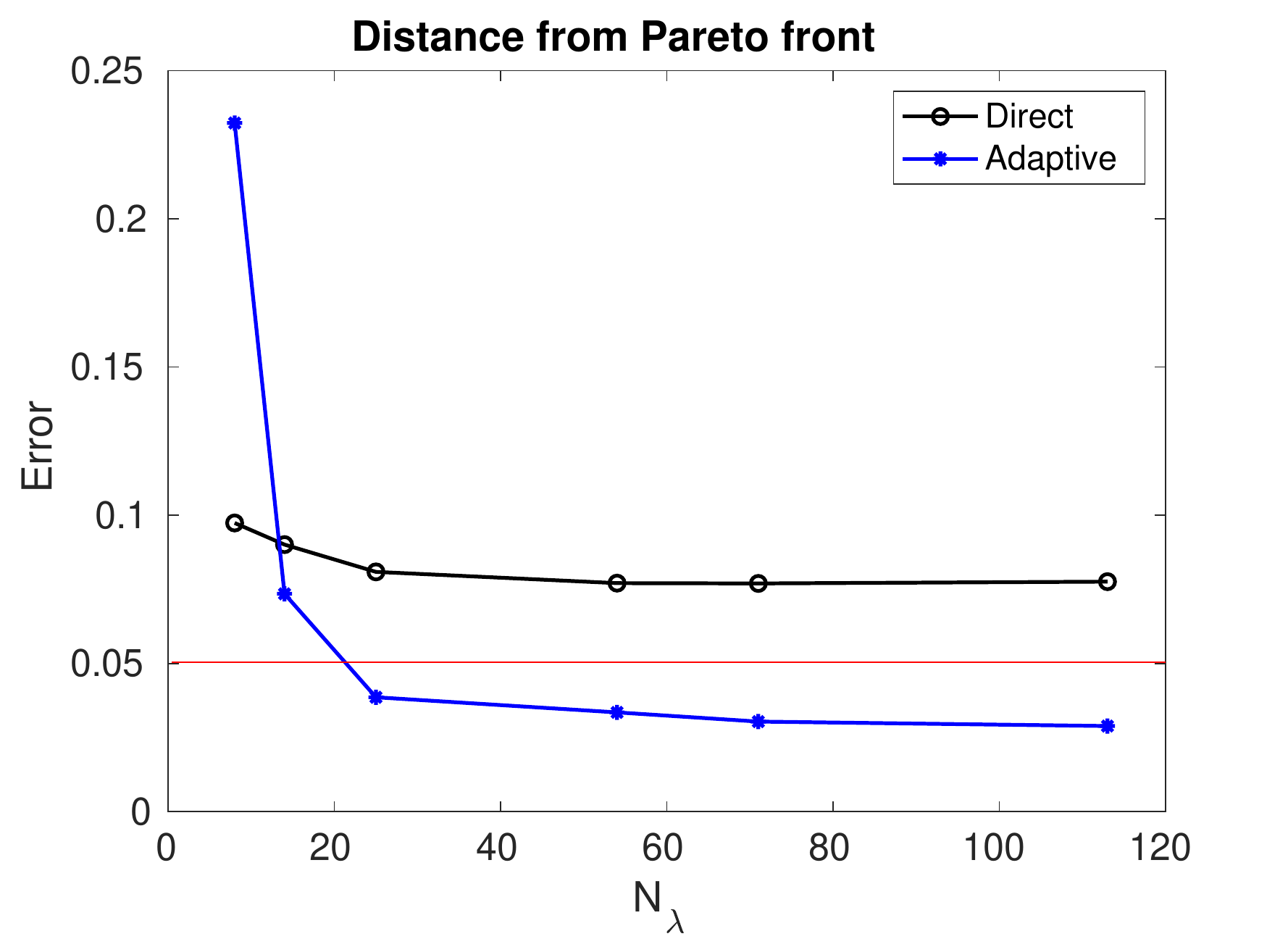}
    \caption{Test 1. The distance from the Pareto front computed by \eqref{eq:distance} is shown for different number of points, for the direct approach (black) and for the adaptive one (blue). The red line indicates the resolution of the Pareto front given by $N_g$.}
    \label{fig:test1_distance}
\end{figure}

\subsection{Test 2: Non-Convex Case}
We consider two non convex functions $\mathcal{G}_1$ and $\mathcal{G}_2$ defined by

\begin{equation}
    \G_1(u)=1 - e^{-(u-1)^2} \qquad \G_2(u)=1 -e^{-(u+1)^2}.
\end{equation}
and  an initial ensemble $U_0\sim \mathcal{U}(-2,2)$ of size $J=50$, and  $\delta=1\cdot 10^{-3}$ leading to  $N_\mathbb{\lambda}=64$. The comparison between the two approaches is shown on $S_G$ in Fig.\ \ref{fig:test2_nl34}. The behavior is similar to the previous case and shows the improvement of the adaptive approach compared with the direct approach.

\begin{figure}[h!]
\hspace{-2.9cm}
\includegraphics[scale=0.4]{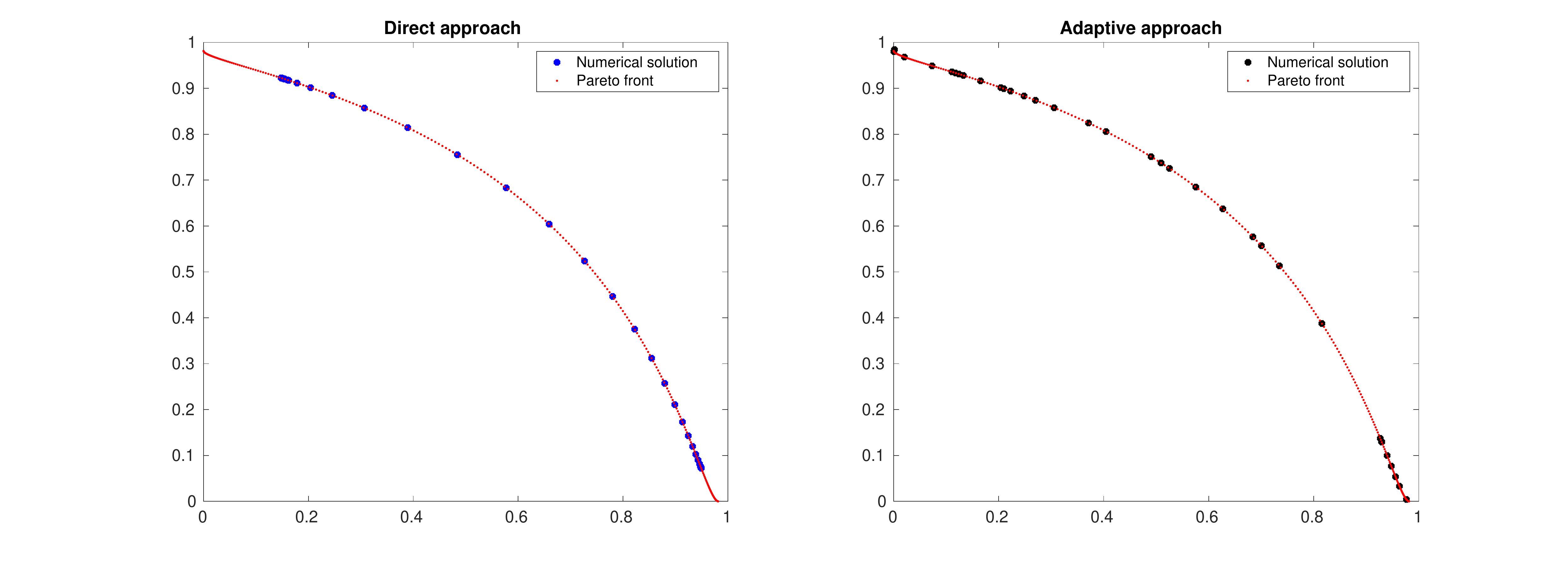}
    \caption{Test 2: Numerical approximation of the Pareto front in the non-convex case, with the direct approach (left) and the adaptive approach (right). The red line is the exact Pareto front.}
    \label{fig:test2_nl34}
\end{figure}


\subsection{Test 3: Multi-dimensional Parameter Space}
We consider $\G_1,\G_2$ two convex functions on $\mathbb{R}^2$ and given by 

\begin{equation}
    \G_1(u_1,u_2)=5(u_1 - 0.1)^2 +   (u_2 - 0.1)^2  \qquad \G_2(u_1,u_2)=(u_1 - 0.9)^2 + 5(u_2 - 0.9)^2.
\end{equation}
The initial ensemble is chosen uniformly distributed  $U_0\sim \mathcal{U}([0,1]^2)$ and we consider $J=30$ particles. We set $\delta=8 \cdot 10^{-4}$ and  $N_\mathbb{\lambda}=68$. A similar  behavior as before is observed in  Fig.\ \ref{fig:test3}. However, the approximation to $S_G$ does not match completely the analytically solution, especially in the region at $x= 2.$ It is assumed that this  is due to the terminal time and we refer to  Fig.\ \ref{fig:test3}-\ref{fig:test3_more} 
where $T_{fin}=50.$  Furthermore, we show the approximation to the set of Pareto points $S_U$ in Fig.\ \ref{fig:test3_ensembles}.   The adaptive choice of sampling $\Lambda$ leads to a relatively sharp resolution of the set $S_U$ compared with the direct approach. The later produces a cloud of points compared to the clusters obtained with the 
adaptive strategy.

\begin{figure}[h!]
	\hspace{-2.8cm}
	\includegraphics[scale=0.42]{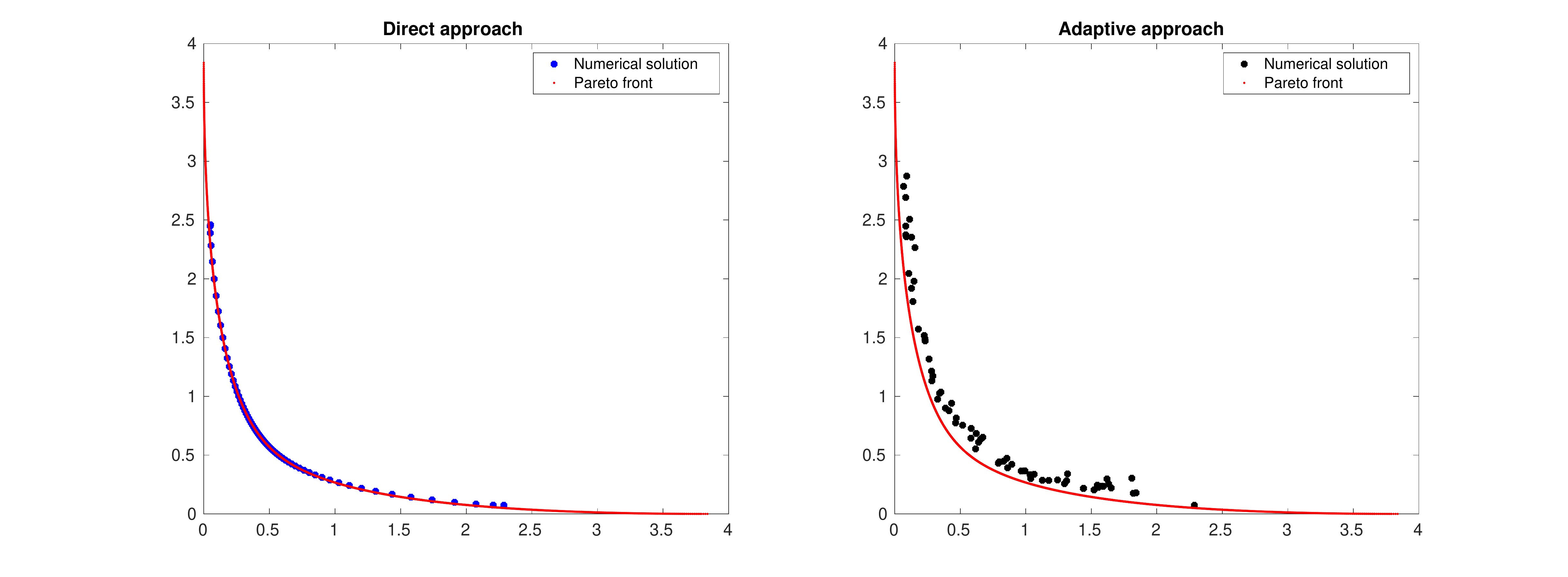}
	\caption{Test 3. Numerical approximation of the Pareto front, with the direct approach(left) and the adaptive approach (right) at $T_{fin}=5$. The red line is the analytical Pareto front.}
	\label{fig:test3}
\end{figure}

\begin{figure}[h!]
\hspace{-2.8cm}
\includegraphics[scale=0.38]{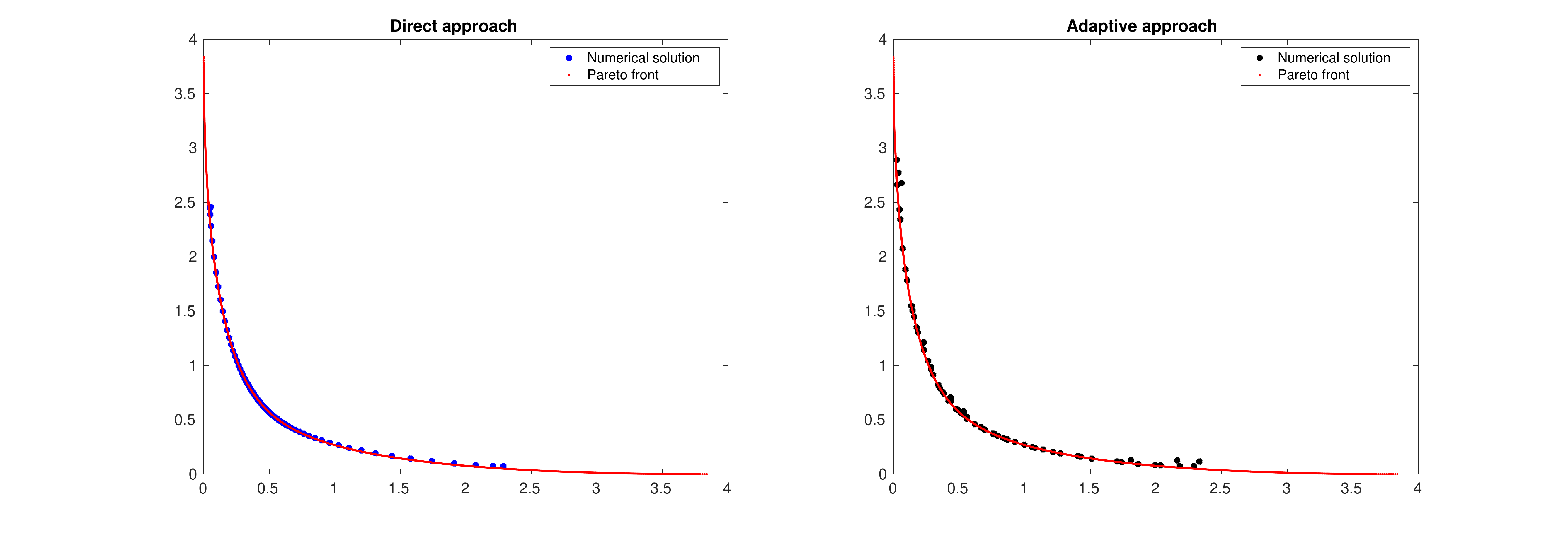}
    \caption{Test 3. Numerical approximation of the Pareto front, with the direct approach(left) and the adaptive approach (right) at $T_{fin}=50$. The red line is the analytical Pareto front.}
    \label{fig:test3_more}
\end{figure}


\begin{figure}[h!]
	\hspace{-0.8cm}
	\includegraphics[width=0.55\linewidth]{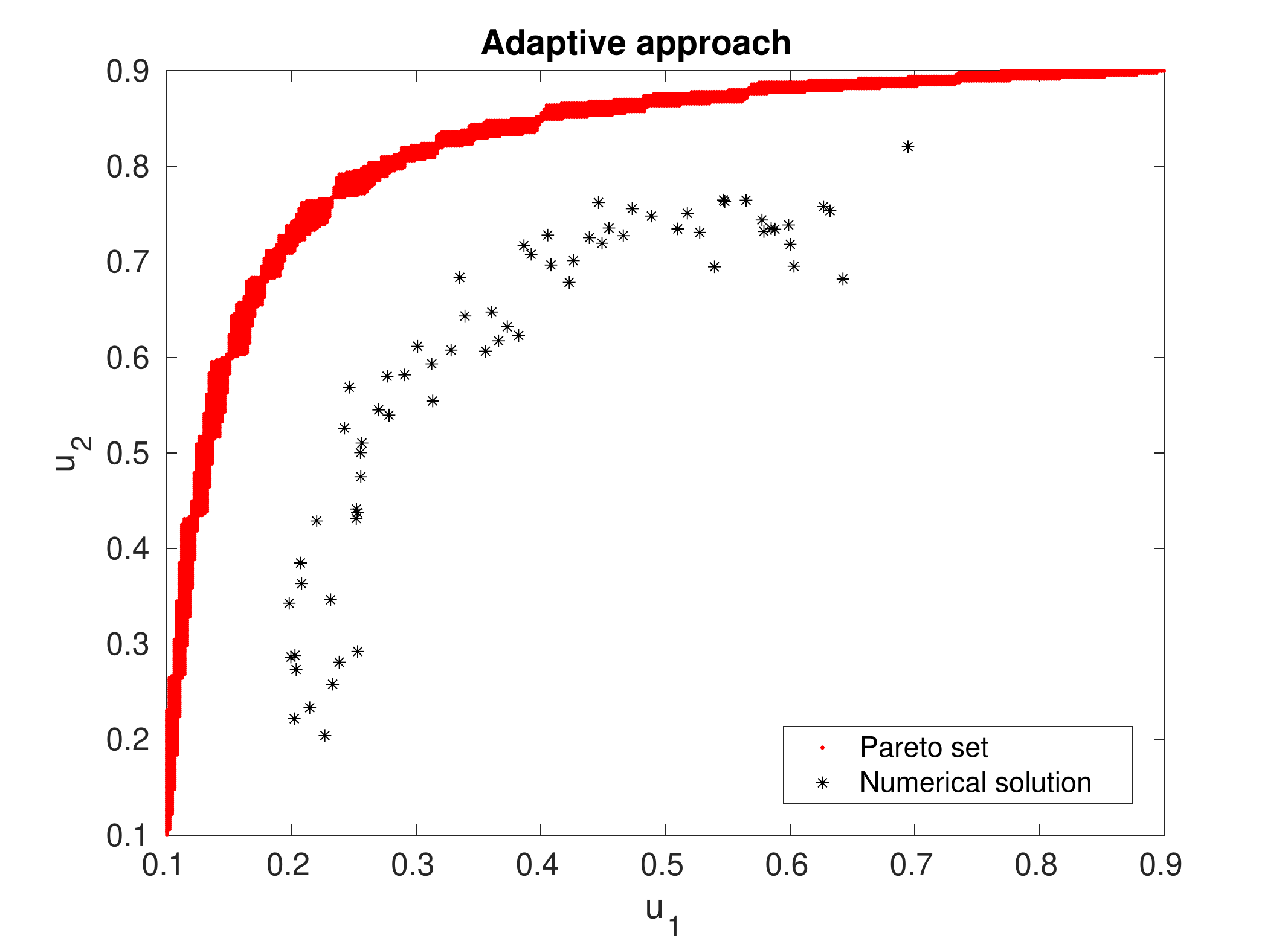}
	\includegraphics[width=0.56\linewidth]{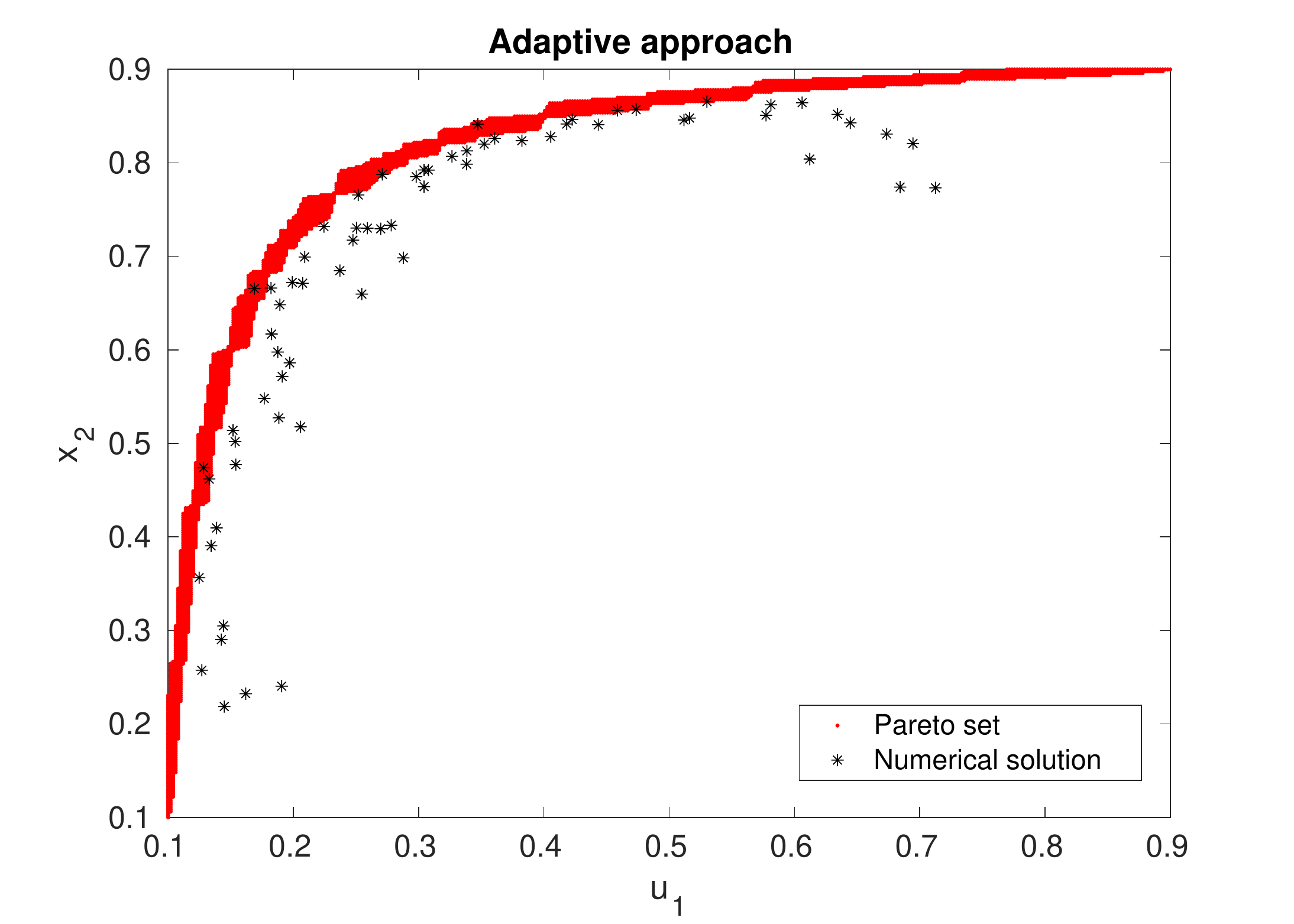}
	\caption{Test 3: Numerical approach of the Pareto set for the adaptive approach for $T_{fin}=5$ (left) and $T_{fin}=50$ (right).}
	\label{fig:test3_set}
\end{figure}


\begin{figure}[h!]
    \includegraphics[width=0.535\linewidth]{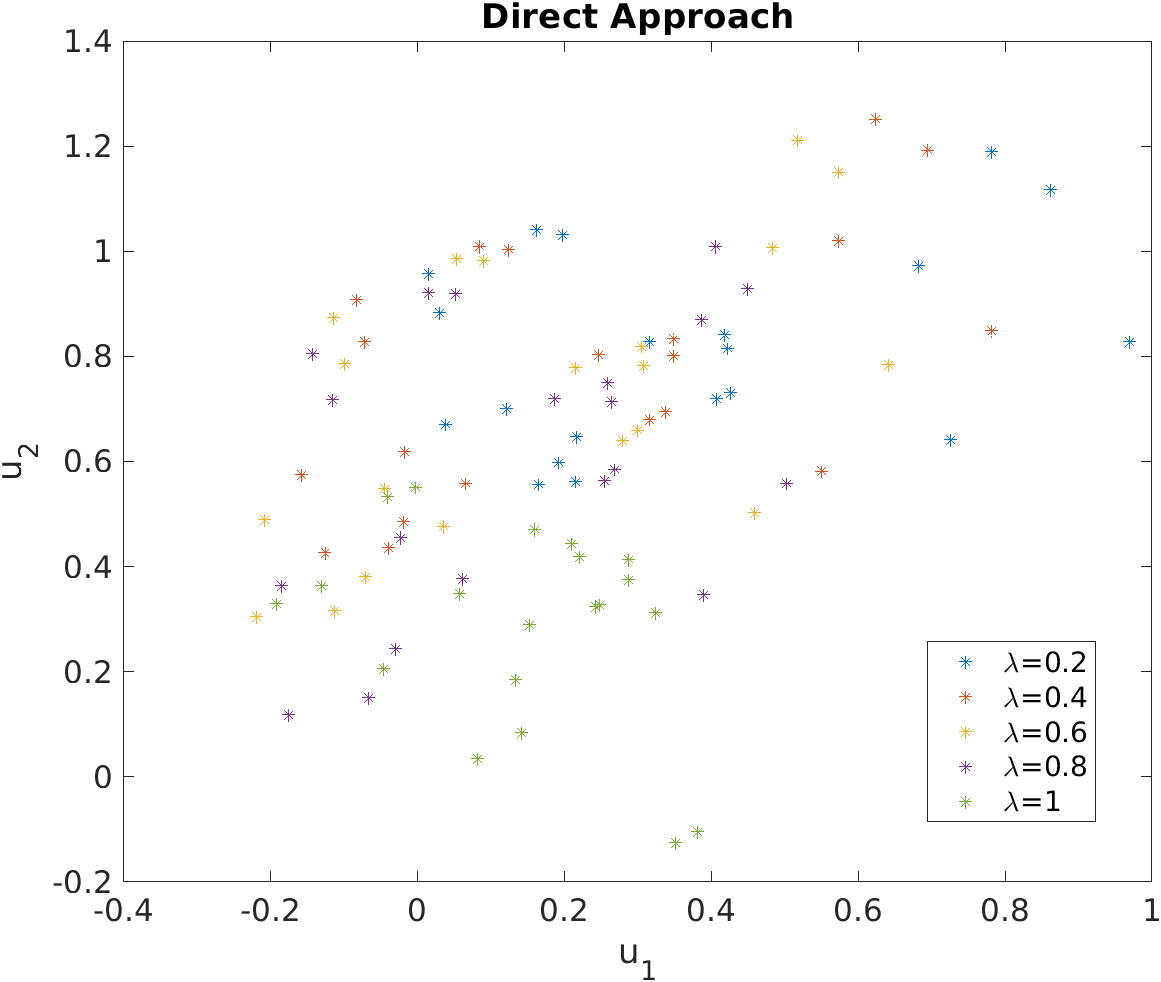}
    \includegraphics[width=0.58\linewidth]{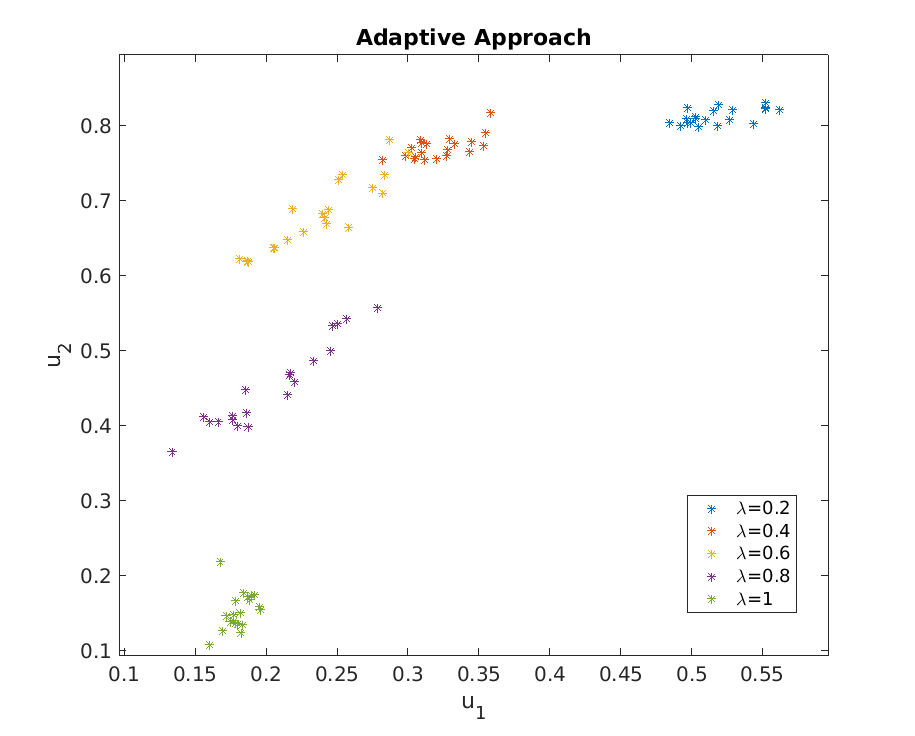}
    \caption{Test 3: Ensemble distribution at $T_{fin}=50$ for different values of $\mathbb{\lambda}$ indicated by color.}
    \label{fig:test3_ensembles}
\end{figure}

\section{Summary}
{The ensemble Kalman filter method has been extended to solve coupled inverse problems. The link to a multi--objective optimization problem has been shown and the analytical properties of the ensemble based method have been investigated. In particular, the mean field equation and their corresponding moment system have been presented and exploited to develop a new adaptive  approach for sampling the Pareto front. Numerical results show the improvement of the adaptive strategy also in the nonlinear case.   }

{\small \subsection*{Acknowledgments}
The authors thank the Deutsche Forschungsgemeinschaft (DFG, German Research Foundation) for the financial support through 20021702/GRK2326,  333849990/IRTG-2379, HE5386/19-2,22-1,23-1 and under Germany's Excellence Strategy EXC-2023 Internet of Production 390621612. 
}
\bibliographystyle{siam}
\bibliography{references}

\end{document}